\documentclass[11pt,A4paper]{article}

\usepackage[left=32mm,right=32mm,top=30mm,bottom=32mm]{geometry}
\usepackage{labelfig}
\usepackage{epsfig}
\usepackage{epstopdf}

\usepackage{xfrac}

\usepackage[percent]{overpic}

\usepackage{color}
\usepackage{amsthm,amsmath,amssymb}
\usepackage{booktabs}
\usepackage{mathpazo}
\usepackage{microtype}
\usepackage{overpic}
\usepackage{bm}
\usepackage{sectsty}
\usepackage[
	pdftitle={PDFTitle},
	pdfauthor={Hugo Parlier},
	ocgcolorlinks,
	linkcolor=linkred,
	citecolor=linkred,
	urlcolor=linkblue]
{hyperref}

\definecolor{linkred}{RGB}{255,102,102}
\definecolor{linkblue}{RGB}{95,158,160}

\usepackage[hang,flushmargin]{footmisc}
\usepackage{enumitem}
\usepackage{titlesec}
	\titlespacing{\section}{0pt}{12pt}{0pt}
	\titlespacing{\subsection}{0pt}{6pt}{0pt}
	
\titlelabel{\thetitle.\quad}

\makeatletter 

\long\def\@footnotetext#1{%
\H@@footnotetext{%
\ifHy@nesting 
\hyper@@anchor{\@currentHref}{#1}%
\else 
\Hy@raisedlink{\hyper@@anchor{\@currentHref}{\relax}}#1%
\fi 
}}

\def\@footnotemark{%
\leavevmode 
\ifhmode\edef\@x@sf{\the\spacefactor}\nobreak\fi 
\H@refstepcounter{Hfootnote}%
\hyper@makecurrent{Hfootnote}%
\hyper@linkstart{link}{\@currentHref}%
\@makefnmark 
\hyper@linkend 
\ifhmode\spacefactor\@x@sf\fi 
\relax 
}%

\ifFN@multiplefootnote%
\renewcommand*\@footnotemark{%
\leavevmode 
\ifhmode 
\edef\@x@sf{\the\spacefactor}%
\FN@mf@check 
\nobreak 
\fi 
\H@refstepcounter{Hfootnote}%
\hyper@makecurrent{Hfootnote}%
\hyper@linkstart{link}{\@currentHref}%
\@makefnmark 
\hyper@linkend 
\ifFN@pp@towrite 
\FN@pp@writetemp 
\FN@pp@towritefalse 
\fi 
\FN@mf@prepare 
\ifhmode\spacefactor\@x@sf\fi 
\relax%
}%
\fi 

\makeatother 

\theoremstyle{plain}
\newtheorem{theorem}{Theorem}[section]
\newtheorem{proposition}[theorem]{Proposition}

\newtheorem{corollary}[theorem]{Corollary}

\theoremstyle{definition}

\newtheorem{remark}[theorem]{Remark}

\newcommand{\R}{{\mathbb R}}

\newcommand{\Hyp}{{\mathbb H}}
\newcommand{\N}{{\mathbb N}}

\newcommand{\Z}{{\mathbb Z}}

\newcommand{\arccosh}{{\,\rm arccosh}}
\newcommand{\diam}{{\rm diam}}

\sectionfont{\large \bfseries}
\subsectionfont{\normalsize}

\long\def\symbolfootnote[#1]#2{\begingroup%
\def\thefootnote{\fnsymbol{footnote}}\footnote[#1]{#2}\endgroup}

\def\blfootnote{\xdef\@thefnmark{}\@footnotetext}

\begin{document}

{\Large \bfseries  Chromatic numbers for the hyperbolic plane and discrete analogs}

{\large Hugo Parlier\symbolfootnote[1]{\normalsize Research supported by Swiss National Science Foundation grant number PP00P2\textunderscore 153024}, Camille Petit \symbolfootnote[7]{\normalsize Research supported by Swiss National Science Foundation grant number 200021\textunderscore 153599\\
{\em 2010 Mathematics Subject Classification:} Primary: 05C15, 30F45. Secondary: 05C63, 53C22, 30F10. \\
{\em Key words and phrases:} chromatic numbers, hyperbolic plane, trees}
}

{\bf Abstract.} 
We study colorings of the hyperbolic plane, analogously to the Hadwiger-Nelson problem for the Euclidean plane. The idea is to color points using the minimum number of colors such that no two points at distance exactly $d$ are of the same color. The problem depends on $d$ and, following a strategy of Kloeckner, we show linear upper bounds on the necessary number of colors. In parallel, we study the same problem on $q$-regular trees and show analogous results. For both settings, we also consider a variant which consists in replacing $d$ with an interval of distances. 

\vspace{1cm}

\section{Introduction} \label{s:introduction}

The geometry of the hyperbolic plane $\Hyp$ appears in a large variety of mathematical contexts and, as such, has been extensively studied. Nonetheless, there are certain combinatorial questions about $\Hyp$ about which not much is known. We're mainly interested in a type of chromatic number for $\Hyp$.

The celebrated Hadwiger-Nelson problem is the search for the minimal number of colors necessary to color the Euclidean plane such that any two points at distance $1$ are colored differently. This chromatic number, denoted by $\chi(\R^2)$, has been known to between $4$ and $7$ for a half-century, but significant progress has eluded mathematicians for decades (see \cite{Soifer11} for details). This can - and has - been studied for other metric spaces such as $\R^n$ \cite{SoiferBook}. The choice of distance $1$ for a Euclidean space is not important thanks to homotheties. In general however, the chromatic number of a metric space will depend on a choice of $d>0$ and colorings are required to have points at distance exactly $d$ colored differently. 

For $\Hyp$ the choice of $d$ is important and we denote the $d$-chromatic number $\chi(\Hyp,d)$. As suggested in \cite{Kloeckner}, letting $d$ grow and studying the growth of $\chi(\Hyp,d)$ could be compared to the study of $\chi(\R^n)$ for growing $n$ which is known to grow exponentially in $n$ (see \cite{Taha-Kahle,SoiferBook} and references therein). The analogy will only be interesting if $\chi(\Hyp,d)$ is shown to grow with $d$ but that's not known to be true. The same proof as for the Euclidean plane \cite{Kloeckner} gives a universal lower bound of $4$ for $\chi(\Hyp,d)$ and that seems to be the extent of the current state of knowledge for lower bounds. 

Our focus point will be on upper bounds. The following theorem summarizes some of our concrete results.
\begin{theorem}
For $d\leq 2 \log(2) \approx 1.389...$ we have
$$\chi(\Hyp,d) \leq 9.$$
For $d \leq 2 \log(3)$ 
$$\chi(\Hyp,d) \leq 12.$$
For $d\geq 2 \log(3)$ the following holds:
$$
\chi(\Hyp,d) \leq 5 \left( \left\lceil \frac{d}{\log(4)} \right\rceil +1\right).
$$
\end{theorem}
Our methods and proof follow the general strategy of using a "hyperbolic checker board", a method outlined in \cite{Kloeckner} and attributed to Sz\'ekely. Kloeckner \cite{Kloeckner} explains how to get a linear upper bound (in $d$) and asks many interesting questions. Our bounds answer one of the questions (Problem {\bf R}). More importantly, we optimize the strategy (Theorems \ref{thm:chromahypupper1} and \ref{thm:chromahypupper2}) and provide some missing arguments. It is these additional details that allow for improved bounds for both small $d$ and larger $d$ (Theorems \ref{thm:smalld}, \ref{thm:larged} and Proposition \ref{prop:effectivesmall}). Note that these questions could also be asked more generally for any hyperbolic surface, but, as was shown by the authors in \cite{ParlierPetit}, the bounds are very different and grow exponentially in $d$. 

We note that for small $d$, it seems very unlikely that the bounds we provide are close to optimal. This is illustrated in Proposition \ref{prop:funddom} where we show how to use a fundamental domain to bound $\chi(\Hyp,d)$ by $8$, but it only works for certain values of $d$. 

When studying the problem of the hyperbolic plane, we started looking for discrete analogs that might help us understand the structure of subgraphs of $\Hyp$ that occur for larger $d$ and that have hyperbolicity properties. This lead us to looking at infinite $q$-regular trees. Although they are bipartite, we can look at their $d$-chromatic number and study it analogously to $\Hyp$. The upper bounds we obtained are close in spirit to those of $\Hyp$ and obtained by a similar method. We synthesize them as follows (see Theorem \ref{thm:puretree}).
\begin{theorem}
If $d$ is odd then
$$\chi(T_q,d) = 2.$$
If $d$ is even then
$$\chi(T_q,d) \leq (q-1) (d+1).$$
\end{theorem}

Lower bounds seem difficult, just like for $\Hyp$. One way of obtaining lower bounds is using a type of clique number, here the maximal number of points at pairwise distance $d$.  For even $d$ this clique number is always $q$ (see Proposition \ref{prop:cliqueq}) which is quite far from our upper bounds. We can improve on that, but only slightly, by producing a generalized Moser spindle (Proposition \ref{prop:moserq}). This gives a lower bound of $q+1$. Nonetheless, we know this lower bound is not optimal as by an extensive computer search we found that $\chi(T_3,8) \geq 5$ (see Remark \ref{rem:ammar}). As for $\Hyp$, the combinatorics seem to get out of hand pretty quickly. 

A property shared by both $\Hyp$ and $T_q$ is that both are natural homogeneous Gromov hyperbolic spaces. In particular, they have thin triangles by which we mean that geodesic triangles with long sides look roughly like tripods (and for $T_q$ they {\it are} tripods). This suggests that an interval chromatic problem might be relevant. In this adaptation, we fix an interval $[d,cd]$ with $d>0$ and $c>1$. We ask that points that have distances that lie in $[d,cd]$ be colored differently. 

Kloeckner \cite{Kloeckner} points out that for the Euclidean plane this interval chromatic number grows like $c^2$ for fixed $d$ and growing $c$ and asks whether
$$
\lim_{c\to \infty} \frac{\chi(\R^2, [d,cd])}{c^2}
$$
exists. He states a purposefully vague interval chromatic problem for the hyperbolic plane (Problem {\bf Z} from \cite{Kloeckner}). We're able to show the following results (Theorems \ref{thm:inthypupper} and \ref{thm:inthyplower}):

\begin{theorem}
For sufficiently large $d$, the quantity $\chi(\Hyp,[d,cd])$ satisfies
$$
2 \, e^{\frac{cd-1}{2}} < \chi(\Hyp, [d,cd]) < 2 \left(2 e^{\frac{cd-1}{2}} + 1\right)(cd +1).
$$
\end{theorem}

For $T_q$, using the same techniques, we show the following (Theorems \ref{thm:inttreeupper} and \ref{thm:inttreelower}).

\begin{theorem}
The quantity $\chi(T_q,[d,cd])$ satisfies
$$
q (q-1)^{\lfloor \frac{cd}{2} \rfloor  - \lceil\frac{d}{2} \rceil}\leq \chi(T_q,[d,cd]) \leq (q-1)^{\lfloor \frac{cd}{2} +1 \rfloor} (\lfloor cd \rfloor+1).
$$
\end{theorem}
The lower bounds in both theorems above come from lower bounds on the (interval) clique numbers. 

{\bf Acknowledgements.}

We heartily thank Ammar Halabi for graciously writing the code necessary to test chromatic numbers for regular trees. Remark \ref{rem:ammar} is thanks to him. 

\section{Preliminaries}

\subsection{Chromatic numbers of metric spaces}

The chromatic number of a graph $G$ is the minimal number of colors needed to color the vertices of a graph such that any two adjacent vertices are of different colors. 

Given a metric space $(X,\delta)$ and a number $d>0$, we define the chromatic number $\chi((X,\delta),d)$ relative to $d>0$ to be the minimal number of colors needed to color all points of $X$ such that any $x,y \in X$ with $\delta(x,y) = d$ are colored differently. We'll sometimes refer to the $d$-chromatic number of $(X,\delta)$. One can define the chromatic number of a metric space via the chromatic number of graphs as follows. Given a metric space $(X, \delta)$ and a real number $d>0$, we construct a graph $G(\{X,\delta\},d)$  with vertices points of $X$ and an edge between points if they are exactly at distance $d$. 

We'll refer to the above chromatic numbers as being {\it pure} chromatic numbers, as opposed to the notion we'll introduce now.

One variant on the pure chromatic number is to ask that points that lie at a distance belonging to a given set be of different colors. An example of this is the chromatic number of $G^k$ power of a graph $G$. This is equivalent to asking that any two vertices at distance belonging to the set $\{1,\hdots,k\}$ be colored differently. More generally, for a metric space $(X,\delta)$ and a set of distances $\Delta$, the $\Delta$-chromatic number $\chi((X,\delta),\Delta)$ is the minimal number of colors necessary to color points of $X$ such that any two points at distance belonging to $\Delta$ are of a different color. As above, this can be seen as the chromatic number of a graph $G(\{X,\delta\}, D)$ where vertices are points of $X$ and edges belong to $D$. We'll be particularly interested in this problem when $D$ is an interval $[a,b]$. We'll refer to these quantities as interval chromatic numbers.

A straightforward way of obtaining a lower bound for chromatic numbers of graphs is via the {\it clique number} which is the order of the largest embedded complete graph. The clique number $\Omega(G)$ clearly satisfies $\Omega(G) \leq \chi(G)$. Similarly we define $\Omega((X,\delta),d)$, resp. $\Omega((X,\delta),\Delta)$, to be the size of the largest number of points of $X$ all pairwise at distance exactly $d$, resp. all at distance lying in $\Delta$.

\subsection{Our metric spaces}

The two types of metric spaces we'll work with are the hyperbolic plane $\Hyp$ and $q$-regular trees (for $q\geq 3$). The unique infinite tree of degree $q$ in every vertex will be denoted $T_q$. Both are viewed as metric space, $\Hyp$ with the standard Poincar\'e metric (an explicit distance formula will be provided below) and $T_q$ as a metric space on vertices obtained by assigning length $1$ to each edge. Although we think of regular trees as a type of discrete analog of the hyperbolic plane, note that the two metric spaces are not even quasi-isometric to one another.

In the next section we'll briefly describe a metric relationship between $\Hyp$ and $T_q$, namely a quasi-isometric embedding of $T_q$ into $\Hyp$. It is provided for motivational purposes and, as it will not be used in the sequel, it can be skipped by the less interested reader.

\subsection{Locally flat models of the hyperbolic plane and geometrically embedded trees}

We describe a locally "flat" model of the hyperbolic plane which is quasi-isometric to $\Hyp$ into which regular trees geometrically embed.

One way of constructing a space which shares properties with hyperbolic plane is to paste together copies of an equilateral Euclidean triangle $\tau$ with sides lengths $1$. 

To do so, fix an integer $n\geq 6$ and construct a simply connected space as follows. Starting with a base copy of $\tau$, paste $n$ copies of $\tau$ around each vertex to obtain a larger simply connected shape. Then we repeat the process indefinitely to get an unbounded simply connected domain which we'll denote $H_n$. 

For example: if $n=6$ then the result is the Euclidean plane. In particular, vertices of copies of $\tau$ map to points of angle $2 \pi$. 

However, for any $n\geq 7$, the set of vertices maps to singular points of angle $\frac{\pi }{3} n$. We note that, for all $n\geq 6$, $H_n$ is a $\mathrm{CAT}(0)$ metric space. 

The following is well-known to experts, but we provide a sketch proof for completeness.

\begin{proposition}
For $n\geq 7$, $H_n$ and $\Hyp$ are quasi-isometric. 
\end{proposition} 

\begin{proof}
To see this, it suffices to construct a quasi-isometry between $\Hyp$ and $H_n$. Consider the cell decomposition of $H_n$ dual to its triangulation: each cell is an $n$-gon with a singularity in its center. As it is dual to a triangulation, the valency in each vertex of this $n$-gon decomposition is three. Denote by $P_n$ a copy of this singular $n$-gon. 

Now consider the unique tiling (up to isometry) of $\Hyp$ by regular $n$-gons of angles $\frac{2\pi}{3}$. Denote by $Q_n$ this hyperbolic $n$-gon we use to tile. Note that both $P_n$ and $Q_n$ have the $n$-th dihedral group $D_n$ as isometry group.

Let $f: H_n \to Q_n$ be any bijective map which, for simplicity, we'll suppose sends the boundary to the boundary and is invariant by the actions of $D_n$. (This is actually not strictly necessary but it simplifies the discussion somewhat.) The map $f$, by compactness of $H_n$ and $Q_n$, is of bounded distortion. 

There are now natural maps between $\Hyp$ and $H_n$ which consists on replacing each regular hyperbolic $n$-gon of the tiling by the singular Euclidean analogue and vice-versa. Points are associated via $f$ and by invariance of $D_n$, coincide with respect to the pasting. The result is a bijection between $\Hyp$ and $H_n$ which is clearly of bounded distortion. 
\end{proof}

The reason we've introduced $H_n$ is that we have the following embedding. By isometric embedding we mean an embedding between metric spaces $(X_1,\delta_1) \hookrightarrow (X_2,\delta_2)$ such that the induced metric on $X_1$ by $X_2$ coincides with the metric $\delta_1$. 

\begin{proposition}\label{prop:embed}
For any $q \leq \lfloor \frac{n}{3} \rfloor$, $T_q$ isometrically embeds into $H_n$. 
\end{proposition}
\begin{proof}
There are two things to prove. The first is that there is an embedding. To do so, we think of $T_q$ as being embedded in the plane (this gives us an orientation at every vertex). 

By vertices of $H_n$ we mean the set of points that are the image of the vertices of the triangles used to construct $H_n$. By edges of $H_n$, we mean the image of the edges of the triangles (all of length $1$). We're going to map vertices and edges of $T_q$ to their counterparts in $H_n$

Take a base vertex $v_0$ of $T_q$ and map it to a base vertex $w_0$ of $H_n$. Now map an edge $e$ of $T_q$ incident to $v_0$ to an edge $e'$ of $H_n$ incident to $w_0$. 

We now map edges incident in $v_0$ to edges incident in $w_0$. Edges around $v_0$ and $w_0$ both have orientations and are ordered relatively to $e$ and $e'$. Following this orientation, edges incident to $v_0$ are mapped to edges incident to $w_0$ ensuring that if any two edges in $P_n$ that are image edges form an angle of at least $\pi$. In other terms, following the order around $w_0$, there are at least two edges between image edges. Note that this was possible thanks to the condition on $q$ and $n$. 

We've now mapped all edges of $T_q$ incident in $v_0$. We now repeat this to map to all vertices distance $1$ from $v_0$, and then inductively, to those at distance $r \geq 2$. This provides us with an embedding $\varphi$.

By construction, the embedding is geometric. Indeed, let $v,w$ be vertices of $\varphi(T_q)$ and let $\gamma$ be the unique simple path between them contained in $\varphi(T_q)$ (image of the unique geodesic in $T_q$). We want to check that $\gamma$ is locally geodesic everywhere. As the space $H_n$ is $\mathrm{CAT}(0)$, this will guarantee that $\gamma$ is the unique geodesic in $H_n$ between $v$ and $w$. To do so we check the angle conditions along $\gamma$. By construction, the angle is $\pi$ along the flat portions of $\gamma$ and at least $\pi$ in every vertex by construction. This proves that our embedding is isometric.
\end{proof}

In terms of chromatic numbers, the isometric embedding above provides the following immediate lower bound. 
\begin{corollary}
For all $d \geq 1$ and $q\geq 3$ and $n$ satisfying $n\geq \lfloor \frac{n}{3} \rfloor$:
$$
\chi(T_q, d) \leq \chi(H_n, d).
$$
\end{corollary}
\section{Pure chromatic number problem}

In this section we'll be concerned with finding upper and lower bounds for the $d$-chromatic number for both the hyperbolic plane and for $q$-trees. We begin with the former.

\subsection{Bounds for the hyperbolic plane}

We'll be using the upper half plane model $\Hyp$ of the hyperbolic plane. The hyperbolic distance formula for $\Hyp = \{ (x,y) \in \R^2 \mid y >0 \}$ can be expressed as
$$
d_{\Hyp}\left((x,y), (x',y')\right) = \arccosh\left( 1+ \frac{(x-x')^2 + (y-y')^2}{2 y y'}\right).
$$

We want to minimally color $\Hyp$ for given $d>0$ such that any two points at distance $d$ are of different colors. 

One quick word about lower bounds for these quantities. Getting a good lower bound via induced complete graphs is futile -  just like for the Euclidean plane the clique number has an upper bound of $3$. And just like in the Euclidean plane, a lower bound of $4$ for any $d$ can be obtained by finding a metric copy of the Moser spindle. It seems likely that one can do better, at least for large $d$, but the combinatorics quickly get out of hand. 

So we focus on upper bounds. The general strategy will always be the same: cover $\Hyp$ with monochromatic regions of diameter less than $d$ and ensure that any two regions of the same color are sufficiently far apart. 

\subsubsection{Construction of a hyperbolic checkerboard}

To color the hyperbolic plane we use a horocyclic checker board constructed as follows. According to \cite{Kloeckner} where it is used to color the hyperbolic plane, this construction is due to Sz\'ekely. Unfortunately some of the key details in \cite{Kloeckner} are incorrect so for completeness we provide a detailed construction.

The method consists in tiling the hyperbolic plane by isometric rectangles where two sides of the rectangle are sub arcs of geodesics with a common point at infinity (so called ultra parallel curves) and the other two sides are horocyles surrounding that same point at infinity. 

To do so formally, we begin by cutting the hyperbolic plane into infinite strips bounded by horocycles as follows. We fix $h>0$ and set $j\in \Z$, the strip $S_j(h)$ to be the set
$$S_j(h) := \left\{ (x,y) \in \Hyp \mid y \in [e^{jh}, e^{(j+1) h}[ \right\}.
$$
We'll sometimes call $S_j(h)$ a {\it strata}. In this model the horizontal lines $y = e^{jh}$ are horocycles around $\infty$ and $h$ is the distance between the horocycles $y = e^{jh}$ and $y = e^{(j+1)h}$. 

Roughly speaking, we now subdivide the strips $S_j(h)$ by cutting them along vertical lines (geodesics). We fix a value $w>0$ and cut along vertical geodesic segments in a way that the two endpoints of the base of each rectangle are at distance $w$. As we don't want the rectangles to overlap even in their boundary, we choose that the vertical geodesic segments belongs to the rectangle on its right. There is some choice is doing the above procedure but we will not make use of that choice in any way (and in fact trying to use this horizontal parameter is tricky). One possible choice leads to the following rectangles which for fixed $h,w$ we can label with elements of $\Z^2$:
$$
R_{i,j}(h,w) := \left\{ (x,y) \in \Hyp \mid x \in [ r j e^{ih}, r(j+1) e^{ih} [,\, y \in  [e^{jh}, e^{(j+1) h}[ \right\}
$$
where
$$r := \sqrt{2(\cosh(w) - 1)}.$$

\begin{figure}[h]
{\color{linkblue}
\leavevmode \SetLabels
\L(0.68*.73) $\arccosh\left( 1+ \frac{\cosh(w)-1}{e^{2h}}\right)$\\
\L(0.68*.5) $h$\\
\L(0.68*.37) $w$\\
\L(0.34*.33) $(0,1)$\\
\L(0.34*.6) $(0,e^h)$\\
\L(0.58*.33) $(r,1)$\\
\L(0.58*.6) $(r,e^h)$\\
\endSetLabels
\begin{center}
\AffixLabels{\centerline{\epsfig{file =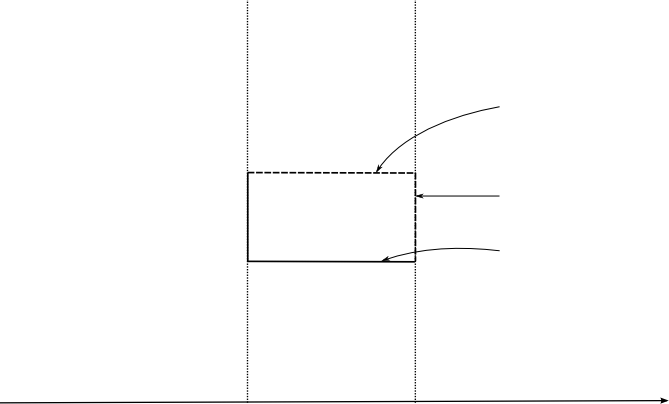,width=10.0cm,angle=0} }}
\vspace{-30pt}
\end{center}
\caption{The rectangle $R_{0,0}(h,w)$} \label{fig:Rectangle}
}
\end{figure}
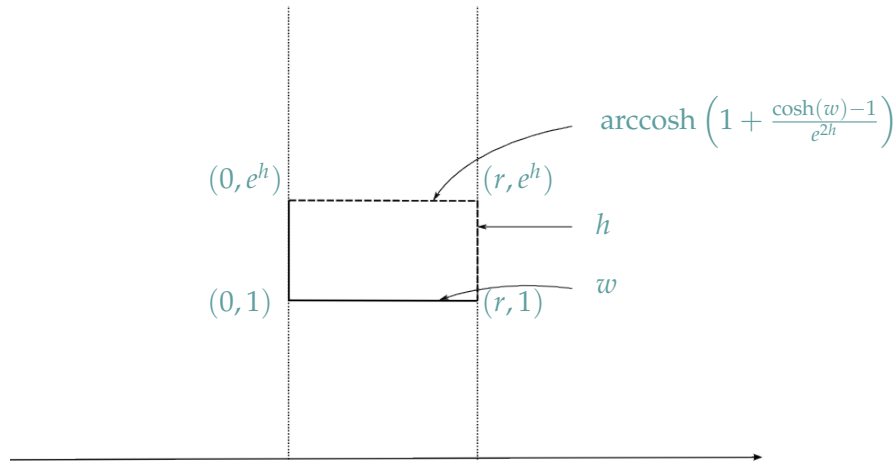

If the above formula is slightly confusing, it's useful to keep in mind one particular copy of the rectangle since all of them are isometric. The rectangle $R_{0,0}(h,w)$ has its four vertices given by the points $(0,1)$, $(r,1)$, $(0, e^h)$ and $(r,e^{h})$. Although the base points of the rectangle are at distance $w$, the upper corners are closer to each other and their distance is in fact
$$
\arccosh\left( 1+ \frac{\cosh(w)-1}{e^{2h}}\right).
$$

Understanding the geometry of the rectangles is key in our argument and we want to understand the diameter of the (closed) rectangle.

Via a simple variational argument, the diameter is realized by some pair of points that lie in the corners. As discussed above, the distance between the upper corners is smaller than the distance between points on the base but one possibility is that the bottom corners realize the diameter and in fact for fixed $w$ and small enough $h$, this is the case. The other possibility is that opposite corners realize the diameter and their distance is 
$$
\arccosh\left( 1+ \frac{2(\cosh(w)-1)+ (e^h-1)^2}{2 e^h}\right).
$$
Again, if $h$ is sufficiently large, the above value will be the diameter. 

To see the above observations, it suffices to look at the distance formula for a pair of points $(0,y)$ and $(r,y')$. We think of $y'$ as being a variable beginning at $y'=y$. Now as $y'$ increases, their distance begins by decreasing until eventually reaching a minimum and then increasing towards infinity. Thus there is a certain value of $y'>y$ for which the distance between $(0,y)$ and $(r,y')$ is exactly that of the distance between $(0,y)$ and $(r,y)$. This shows that the pair of points that realize the diameter is either the base or the diagonal and this depends on how large $h$ is. All in all, we've shown the following.

\begin{proposition}
The rectangle $R_{i,j}(w,h)$ satisfies
$$
\diam\left( R_{i,j}(w,h) \right) = \max   \bigg\{ w, \arccosh\left( 1+ \frac{2(\cosh(w)-1)+ (e^h-1)^2}{2 e^h}\right) \bigg\}.
$$
\end{proposition}

We also need to get a handle on the distance between consecutive rectangles in a stratum (so rectangles $R_{i,j}(w,h)$ and $R_{i',j}(w,h)$ for $i'>i$). By the same considerations as above the distance will be realized (in the closure of the rectangles) by the upper right corner of $R_{i,j}(w,h)$ and the upper left corner of $R_{i',j}(w,h)$. The distance formula gives us
$$
d_{\Hyp} \left( R_{i,j}(w,h), R_{i',j}(w,h) \right) = \arccosh \left( 1+ \frac{(i-i')^2 (\cosh(w)-1)}{e^{2h}}\right).
$$
With this in hand we can construct a well-adapted checker board in function of the parameter $d$. The method is to use a checkerboard with rectangles of diameter $\leq d$.  We color stratum by stratum cyclically using the above formula to ensure that if rectangles are horizontally sufficiently far apart, they can be colored the same way. We then need to repeat the above process with completely new colors until the strata are sufficiently far apart (see Figure \ref{fig:distancerectangle}). This requires exactly $\lceil \frac{d}{h} \rceil + 1$ strata to be colored before repeating the procedure. 

This leads to the following general statement that we state as a theorem.

\begin{theorem}\label{thm:chromahypupper1}
The $d$-chromatic number of the hyperbolic plane satisfies 
$$
\chi(\Hyp,d) \leq (k + 1) \left( \left\lceil \frac{d}{h} \right\rceil +1 \right)
$$
for any $w,h$ that satisfy
$$
 \max   \bigg\{ w, \arccosh\left( 1+ \frac{2(\cosh(w)-1)+ (e^h-1)^2}{2 e^h}\right) \bigg\} \leq d
$$
and $k$ is the smallest integer satisfying
$$
k \geq e^h \sqrt{\frac{\cosh(d)-1}{\cosh(w)-1}}.
$$
\end{theorem}
We note that the above condition implies that $k\geq 2$. 

\begin{figure}[h]
{\color{linkblue}
\leavevmode \SetLabels
\L(0.47*.95) $\geq d$\\
\L(0.08*.35) $\geq d$\\
\endSetLabels
\begin{center}
\AffixLabels{\centerline{\epsfig{file =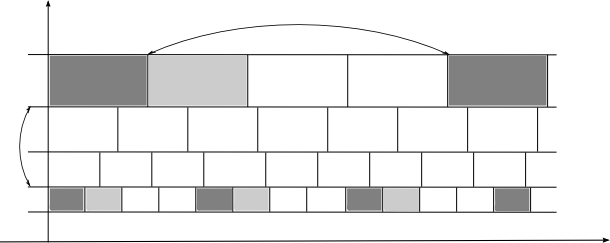,width=12.0cm,angle=0} }}
\vspace{-30pt}
\end{center}
\caption{$k=3$ and $\left\lceil \frac{d}{h} \right\rceil=2$} \label{fig:distancerectangle}
}
\end{figure}

As stated, it is not clear how to optimally apply the theorem. We state a formulation which, although not necessarily practical, will give us the optimal solution for a checkerboard coloring.

Suppose we are given an $h>0$ which satisfies $h<d$. We want to optimize the checkerboard using this fixed $h$. There is now a clear choice of $w$:
$$
w= \min\left\{d, \arccosh\left( \frac{1 + 2 e^h\cosh(d) - e^{2h}}{2}\right)\right\}.
$$
Using this we again have a canonical choice of $k$:
$$
\left\lceil e^h \sqrt{\frac{\cosh(d)-1}{\cosh(w)-1}}\right\rceil.
$$
Everything is now expressed in terms of $h$ and thus we have the following.
\begin{theorem}\label{thm:chromahypupper2}
The $d$-chromatic number of the hyperbolic plane satisfies 
$$
\chi(\Hyp,d) \leq \min_{h<d} \left\{ (k(h) + 1) \left( \left\lceil \frac{d}{h} \right\rceil +1 \right)\right\}
$$
where
$$
w(h):= \min\left\{d, \arccosh\left( \frac{1 + 2 e^h\cosh(d) - e^{2h}}{2}\right)\right\}
$$
and
$$
k(h):= \left\lceil e^h \sqrt{\frac{\cosh(d)-1}{\cosh(w(h))-1}}\right\rceil.
$$
\end{theorem}

We now apply these results to get effective bounds in terms of $d$.

\subsubsection{Bounds on $\chi(\Hyp,d)$ for small $d$}

Note that the above method requires that $k(h), \lceil \frac{d}{h} \rceil >1$. So in particular the method will never allow for a better bound than $9$ on the chromatic number. We now show that this bounds holds for sufficiently small $d$.

\begin{theorem}\label{thm:smalld}
For $d\leq 2 \log(2) \approx 1.389...$ we have
$$\chi(\Hyp,d) \leq 9.$$
\end{theorem}
\begin{proof}
We'll apply the strategy from Theorem \ref{thm:chromahypupper2}. If we want to bound $\chi(\Hyp,d)$ by $9$, we need to have $\frac{d}{h}  \leq  2$.
With this constraint in hand, we set $h = \frac{d}{2}$ as any larger $h$ can only increase $k$ and the diameter of a rectangle. 

We'll need to set $w(h)$ as in Theorem \ref{thm:chromahypupper2} and this depends on $d$. To determine our choice, we'll need to study the function
$$
\min\left\{d, \arccosh\left( \frac{1 + 2 e^h\cosh(d) - e^{2h}}{2}\right)\right\}
$$
for $h =\frac{d}{2}$.

A straightforward analysis tells us to set 
$$
w(h) = \arccosh\left( \frac{1 + 2 e^h\cosh(d) - e^{2h}}{2}\right)
$$ 
for $d\in ]0,d_0]$,
where $d_0$ is the non zero positive solution to the equation
$$
\frac{1+ 2 e^{\frac{d_0}{2}} \cosh(d_0)}{2}  - e^{d_0} - \cosh(d_0)=0.
$$
The precise value for $d_0$ can be computed:
$$
d_0 = 2\log\left( \frac{(108 + 12 \sqrt{69})^\frac{1}{3} + \frac{12}{(108 + 12 \sqrt{69})^\frac{1}{3}}}{6}\right)\approx 0.56...
$$
For $d$ in this interval we can take $k =2$ as the following inequality is satisfied
\begin{eqnarray*}
4 &>& e^{2h} \frac{\cosh(d)-1}{\cosh(w(h))-1}\\
& = & 2 \, e^d \frac{\cosh(d)-1}{- e^d + 2 e^{\frac{d}{2}} \cosh(d) +1}.
\end{eqnarray*}
For $d > d_0$ we are required to set $w(d)=d$. In order to be able to set $k =2$ we need to satisfy:
$$
2 \geq e^{\frac{d}{2}}
$$
which is true provided
$$
d\leq 2 \log(2)
$$
as desired. 

\end{proof}

Even though we had the previous theorems in hand, the above argument still required a case by case analysis, which can be explained geometrically. The diameter of the rectangle for small $d$ was realized by diagonally opposite points, but for larger $d$ it was realized by the base points. 

We can argue similarly to obtain the following results, which again require a case by case analysis. Note that we needed to argue case by case in terms of $k$ and $\lceil \frac{d}{h} \rceil$ so we only include the upper bounds that work for larger intervals of $d$. The strategy is always the same: we want to bound $\chi(\Hyp,d)$ by $N = (k+1)(m+1)$, so we set $h = \frac{d}{m}$ and we argue as above. As we've treated (very) small $d$ already, the diameter of the rectangle will be generally be the base of the rectangle. This will work for all $d$ that satisfy
$$
d \leq m \log(k).
$$
Now if $N = (a+1)(b+1)$, to get a larger interval will require comparing $a \log(b)$ and $b \log(a)$. 

\begin{proposition}\label{prop:effectivesmall} The chromatic numbers of the hyperbolic plane satisfy the following inequalities for certain $d$:
\begin{itemize}
\item For $d \leq 2 \log(3)$:
$$\chi(\Hyp,d) \leq 12.$$
\item For $d \leq 2 \log(4)$:
$$\chi(\Hyp,d) \leq 15.$$
\item For $d \leq 3 \log(3)$:
$$\chi(\Hyp,d) \leq 16.$$
\item For $d \leq 5 \log(2)$:
$$\chi(\Hyp,d) \leq 18.$$
\end{itemize}
\end{proposition}

The process can be continued to obtain optimal intervals where $\chi(\Hyp,d)$ is bounded by integers of the form $N = (a+1)(b+1)$ where both $a$ and $b$ are greater or equal to $2$.

We now turn our attention to large values of $d$.

\subsubsection{Bounds on $\chi(\Hyp,d)$ for large $d$}

For large values of $d$ we set $w:=d$ and $h:= \log(k)$. Provided $d$ is large enough, our bounds tell us that 
$$
\chi(\Hyp,d) \leq (k+1) \left( \left\lceil \frac{d}{\log(k)} \right\rceil +1\right).
$$
We want to optimize the asymptotic growth of this bound in terms of $d$. The relevant factor is 
$$
\frac{k+1}{\log(k)}
$$
which is minimized for $k = 4$. Note that the above bound, for $k=4$ will hold, by {Theorem \ref{thm:chromahypupper1}}, provided 
$$
d \geq  \arccosh\left( \frac{1 + 2 e^h\cosh(d) - e^{2h}}{2}\right)
$$
which is certainly true for all $d \geq 2$ (a more precise value is true but we've proved better bounds above). We have thus proved the following. 

\begin{theorem}\label{thm:larged}
For $d \geq 2 $ we have 
$$
\chi(\Hyp,d) \leq 5 \left( \left\lceil \frac{d}{\log(4)} \right\rceil +1\right).
$$
\end{theorem}
\begin{remark}
We end this analysis by observing that the same argument tells us that 
$$
\chi(\Hyp,d) \leq 4 \left( \left\lceil \frac{d}{\log(3)} \right\rceil +1\right)
$$
for $d \geq 2$. Although this bound is not asymptotically as good as the one in the theorem above, for certain $d$ up until approximately $143$, it provides a stronger estimate. This illustrates the touch and go aspect of the checkerboard method.
\end{remark}

\subsubsection{Using a fundamental domain}

In this section, we briefly remark that there are certain $d$ for which we can bound $\chi(\Hyp,d)$ by $8$. The method is really a hyperbolic analogue of the classical $7$ upper bound on the chromatic number of the Euclidean plane. We provide it to illustrate the current lack of a monotonic method: one might expect that 
$$
\chi(\Hyp,d) \leq \chi(\Hyp,d')
$$
provided $d'< d$ but it seems like a tricky question. 

The coloring is based on tilings that appear when studying Klein's quartic in genus $3$. We'll describe it in simple terms, and show how it's an adaptation of the $7$ upper bound for the Euclidean plane. 

One way of describing the classical Euclidean coloring (for $d=1$) is as follows. Take a tiling of $\R^2$ by a set of regular hexagons of diameter $<1$ (say $0.99$). Now consider the dual graph to this tiling. We fix a base tile and associate to all of its points color $1$.  We color each of the adjacent hexagons colors $2$ to $7$. We now describe how to color all remaining hexagons. From a vertex $u$ of the dual graph, we travel along any edge and then travel along the unique edge at oriented angle $\frac{2\pi}{3}$ to reach a new vertex $v$. From $v$ we then travel along the unique edge at oriented angle $-\frac{2\pi}{3}$ to reach a new vertex $w$ and we color $w$ the same color as $u$. A standard argument tells us that we've colored the entire plane like this.

We adapt this method as follows: we take a regular hyperbolic heptagon $H$ with all angles equal to $\frac{2\pi}{3}$. There is a unique such heptagon and it can be decomposed into $7$ triangles $T$ of angles $\frac{\pi}{3}, \frac{\pi}{3}$ and $\frac{2\pi}{7}$. The diameter of $H$ can be computed using standard hyperbolic trigonometry and it has a value of slightly more than $1.22$. 

We now consider a standard tiling of $\Hyp$ by copies of $H$. Fixing a base copy, we color all points of $H$ the same color. Each of the $7$ surrounding heptagons are given a different color. And we've colored a shape $O$ consisting of $8$ copies of $H$. To describe how to color all other heptagons, we argue using the dual graph. Here the edges of the dual graph meet at angles multiples of $\frac{2\pi}{7}$. From a vertex $u$ of the dual graph, we travel along any edge and then travel along the unique edge at oriented angle $\frac{4\pi}{7}$ to reach a new vertex $v$. From $v$ we then travel along the unique edge at oriented angle $-\frac{4\pi}{7}$ to reach a new vertex $w$ and we color $w$ the same color as $u$. As above, this colors the entire hyperbolic plane. 

Of course this won't work for all $d$. We choose $d\geq 1.22$ to ensure that its bigger than the diameter of the heptagons but we also need to choose $d$ small enough so that translates of the same color are further than $d$. Using standard hyperbolic trigonometry, one can see that any two heptagons are at distance at least $\approx 1.77$. The result of all of this is the following proposition. 

\begin{proposition}\label{prop:funddom}
For $d\in [1.22, 1.77]$ we have
$$
\chi(\Hyp,d) \leq 8.
$$
\end{proposition}
\subsection{Bounds for $q$-trees}

Recall that $\chi(T_q, d)$ is the minimum number of colors required to color a $q$-regular tree such that any two vertices at distance $d$ apart are of a different color. A first immediate bound on this quantity is given by Brooks' theorem. Consider the distance $d$ graph associated to $T_q$: it is a regular graph of degree $q (q-1)^{d-1}$ so 
$$
\chi(T_q, d) \leq q (q-1)^{d-1}+1.
$$
We want to do much better and to do so we emulate the method for $\Hyp$ which required coloring strata. We begin by using a horocyclic decomposition of a tree.

\subsubsection{Strata for horocyclic decompositions}

We describe the method which works identically for any $q$-regular tree $T_q$. 

We begin by choosing a base point $x_0 \in T_q$ and choosing an infinite geodesic ray leaving from this point $[x_0,x_1,\cdots]$ (where $d_{T_q}(x_k, x_{k+1}) = 1$). We think of $\eta = [x_0,x_1,\cdots]$ as a {\it boundary point} of $T_q$ (formally a boundary point is an equivalence class of rays but we won't dwell on that here). 

We define the Busemann function associated to $\eta = [x_0,x_1,\cdots]$ as 
$$
h_\eta(x):= \lim_{y\to \eta} \left(d_{T_q}(y,x)- d_{T_q}(y,x_0) \right) \left(=  \lim_{k\to \infty} \left(d_{T_q}(x_k,x)- d_{T_q}(x_k,x_0) \right)\right).
$$

We can now define the strata $S_n$ as being level sets of the function $h_\eta$:
$$
S_n:= \{ x\in T_q \mid h_\eta(x) = n \}, \, n\in \Z.
$$
We note that the strata are, by analogy with the hyperbolic plane, generally called horocyles and can be thought of as circles centered around a point at infinity. Note that $x_0 \in S_0$ but $x_{k} \in S_{-k}$ for all $k\in \N$ (see Figure \ref{fig:horocyclic construction}). 

\begin{figure}[h]
{\color{linkblue}
\leavevmode \SetLabels
\L(0.1*.93) $S_{-1}$\\
\L(0.1*.72) $S_0$\\
\L(0.1*.5) $S_1$\\
\L(0.1*.28) $S_2$\\
\L(0.1*.05) $S_3$\\
\L(0.52*.96) $x_1$\\
\L(0.52*.74) $x_0$\\
\endSetLabels
\begin{center}
\AffixLabels{\centerline{\epsfig{file =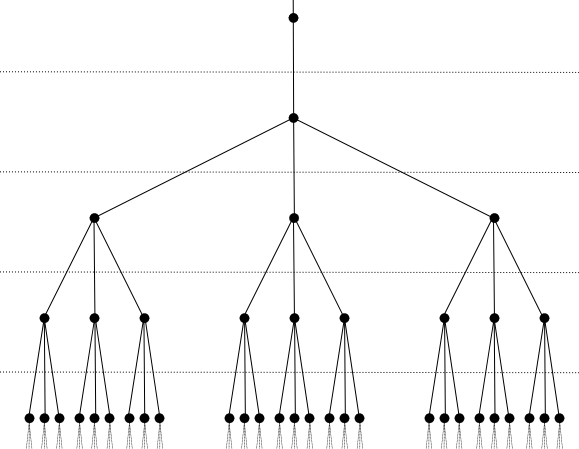,width=10.0cm,angle=0} }}
\vspace{-30pt}
\end{center}
\caption{Horocyclic construction} \label{fig:horocyclic construction}
}
\end{figure}
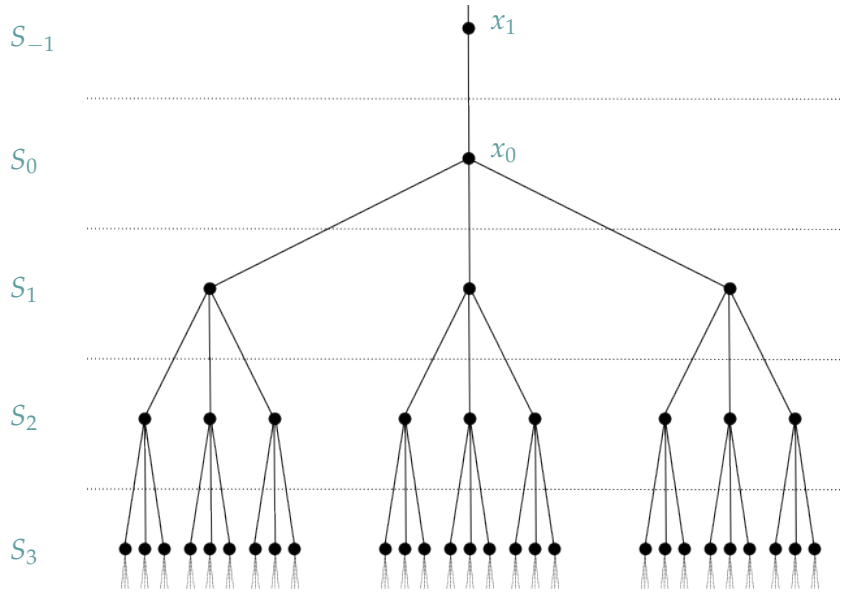

A first observation is that distances between points in the same stratum are always even. More generally, distances are even between points that lie respectively in $S_k$ and $S_{k'}$ with $k$ and $k'$ of same parity. Thus as an immediate corollary of the horocyclic construction we obtain the following.

\begin{corollary}
If $d$ is odd then $\chi(T_q,d) = 2$.
\end{corollary}
\begin{proof}
Clearly $\chi(T_q,d)\geq 2$ and we can color $T_q$ using one color for all points lying in $S_k$ with $k$ even and another for all points lying in $k$ odd.
\end{proof}

When $d$ is even, the problem is not so obvious. 

\subsubsection{Bounds for even $d$}

We now prove upper bounds for even $d$.

\begin{theorem}\label{thm:puretree}
When $d$ is even $\chi(T_q,d) \leq (q-1) (d+1)$.
\end{theorem}
\begin{proof}
We color one stratum at a time and by thinking of the tree as a rooted tree with root at infinity, we bundle vertices on a stratum in terms of their ``ancestors". 

More precisely we'll color all vertices of $S_k$ the same color if they have a common root at distance $\frac{d-2}{2}$. Note that this is possible because any two such vertices are at distance at most $d-1$. 

For a given monochromatic bundle $B$, we now consider all of the other bundles of $S_k$ that have a common ancestor at distance $\frac{d}{2}$. Note there are exactly $q-1$ of these in total (which we'll call a super bundle) and we'll color each bundle a different color requiring $q-1$ colors. We can color all other vertices of $S_k$ with the same $q-1$ colors using the same method as any two vertices lying in different super bundles are distance $>d$ apart. 

Now any two stata $S_k$ and $S_{k'}$ can be colored using the same colors provided $|k-k'|\geq d+1$ so we obtain a coloring with $(q-1) (d+1)$ as required.
\end{proof}

\subsubsection{Lower bounds}

\begin{proposition}\label{prop:cliqueq}
For any even $d\geq 2$, the clique number satisfies $\Omega(T_q,d) = q$.
\end{proposition}
\begin{proof}
The lower bound comes from the following construction. Fix a base vertex: it divides the graph into $q$ branches. Now choosing $q$ vertices, one in each branch, at distance $d/2$ from the base vertex. Any two are at distance $d$, hence the lower bound.

The upper bound works as follows. Suppose by contradiction that there is a clique of size $c> q$ and consider the subgraph of $T_q$ spanned by the distance paths between the $c$ vertices. The vertices of the clique are the leaves in this subgraph $G$. It must contain at least $2$ branching points $v,w$ (vertices of degree at least $3$) as the inner degree is at most $q$. Removing the edges between $v$ and $w$ separates $G$ into two parts $G_v$ and $G_w$. Because the degrees of $v$ and $w$ were at least $3$, both $G_v$ and $G_w$ must contain at least two leaves of $G$. Let $v_1,v_2$, resp. $w_1,w_2$, be leaves of $G_1$, resp. $G_2$. 

We have
$$
d_{T_q}(v_1,v_2) = 2 d_{T_q}(v_1,v)
$$
and
$$
d_{T_q}(w_1,w_2) = 2 d_{T_q}(w_1,w)
$$
but
\begin{eqnarray*}
d_{T_q}(v_1,w_1) & = & d_{T_q}(v_1,v) + d_{T_q}(v,w) + d_{T_q}(w_1,w) \\
 &>& d_{T_q}(v_1,v) +  d_{T_q}(w_1,w) \\
 &\geq& 2 \min\{ d_{T_q}(v_1,v), d_{T_q}(w_1,w) \}
\end{eqnarray*}
and so either $v_1,v_2$ and $w_1$ or $w_1,w_2$ and $v_1$ cannot form a triangle, a contradiction.
\end{proof}
In certain low complexity cases, we can compute the chromatic number explicitly.
\begin{proposition} $\chi(T_3, 2) = 3$.
\end{proposition}
\begin{proof}
Consider the graph $G(T_3,2)$ consisting of vertices of $T_3$ and edges between vertices is they are distance $2$ in $T_3$. 

The $G(T_3,2)$ is pretty easy to visualize. First of all, observe it has two connected components as it is impossible to travel between two vertices at odd distance in $T_3$. By homogeneity, both connected components are isomorphic. 

Take a vertex $v_0$ in $T_3$ and the three vertices it is connected to. Together they form a tripod. The three end vertices of this tripod are all pairwise distance $2$ apart so they form a triangle in $G(T_3,2)$. (Note they are not connected to $v_0$ in $G(T_3,2)$.) In particular $\chi(T_3,2)\geq 3$. 

Now each of these three vertices belongs to $2$ other triangles in $G(T_3,2)$ and the figure repeats itself (see Figure \ref{fig:GT32}). 
\begin{figure}[h]
{\color{linkblue}
\leavevmode \SetLabels
\L(0.52*.94) $1$\\
\L(0.46*.80) $2$\\
\L(0.505*.765) $3$\\
\L(0.65*.82) $1$\\
\L(0.59*.78) $2$\\
\L(0.63*.65) $3$\\
\L(0.545*.60) $1$\\
\L(0.45*.4) $2$\\
\L(0.56*.375) $3$\\
\L(0.65*.41) $1$\\
\L(0.62*.24) $2$\\
\L(0.662*.205) $3$\\
\L(0.574*.155) $1$\\
\L(0.51*.12) $2$\\
\L(0.554*.00) $3$\\
\L(0.38*.64) $1$\\
\L(0.32*.51) $2$\\
\L(0.385*.534) $3$\\
\L(0.365*.33) $1$\\
\L(0.34*.186) $2$\\
\L(0.382*.154) $3$\\
\endSetLabels
\begin{center}
\AffixLabels{\centerline{\epsfig{file =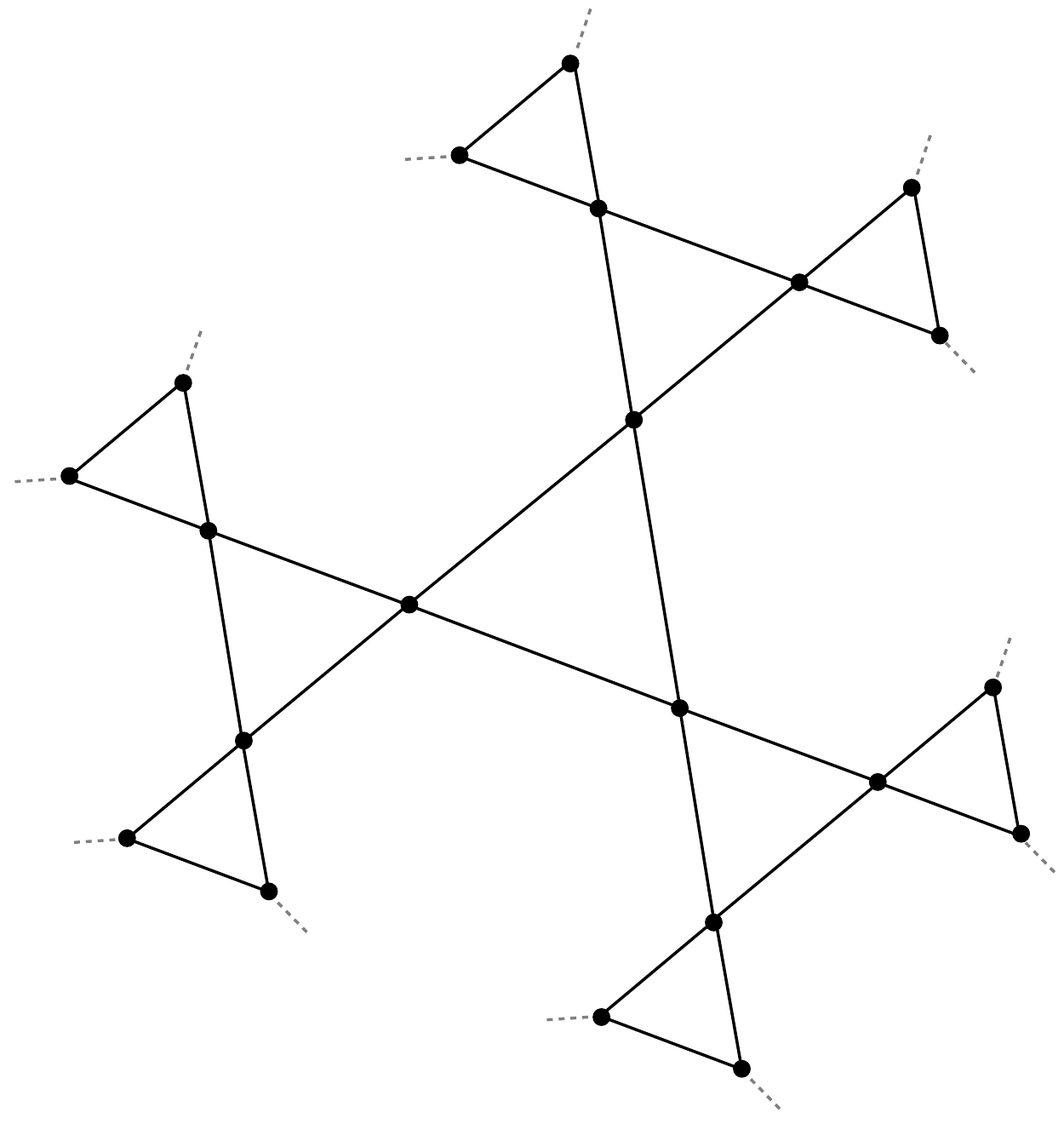,width=6.0cm,angle=0} }}
\vspace{-30pt}
\end{center}
\caption{A connected component of $G(T_3,2)$ and its coloring} \label{fig:GT32}
}
\end{figure}
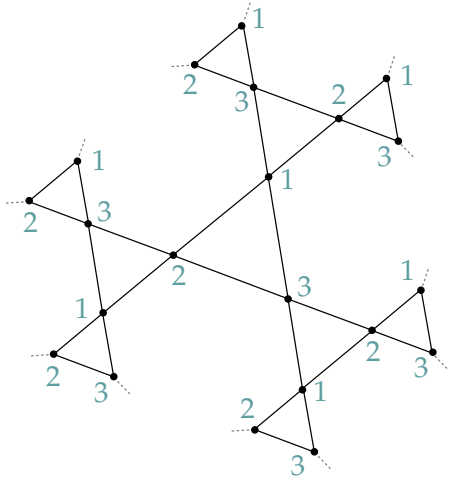
There is a iterative $3$ coloring of this graph by first coloring the vertices of a base triangle, and then those belonging to the triangles attached level by level. The same colors can be used for both connected components and these shows the proposition.
\end{proof}

\begin{proposition}  \label{prop:moserq}
For any even $d\geq 4$ we have
$$\chi(T_q, d) \geq q+1.$$
\end{proposition}

\begin{proof}
We can embed a type of generalized Moser spindle in each of these graphs as follows. 

We take a base vertex $v_0$ and consider two sets of vertices $v_1,\hdots,v_{q-1}$ and $v'_1,\hdots,v'_{q-1}$ all at distance $d$ from $v_0$ and with the following property. Any two $v_i,v_j$, resp. $v'_i,v'_j$, for distinct $i,j$ are at distance $d$. We then consider two additional vertices $v_q$ and $v'_q$ at distance $d$ from one another and such that $v_q$ is distance $d$ from $v_i$ for $i=1,\hdots,q-1$ and $v'_q$ is distance $d$ from $v'_i$ for $i=1,\hdots,q-1$. An example for $q=4$ is illustrated in Figure \ref{fig:MoserSpindle}.

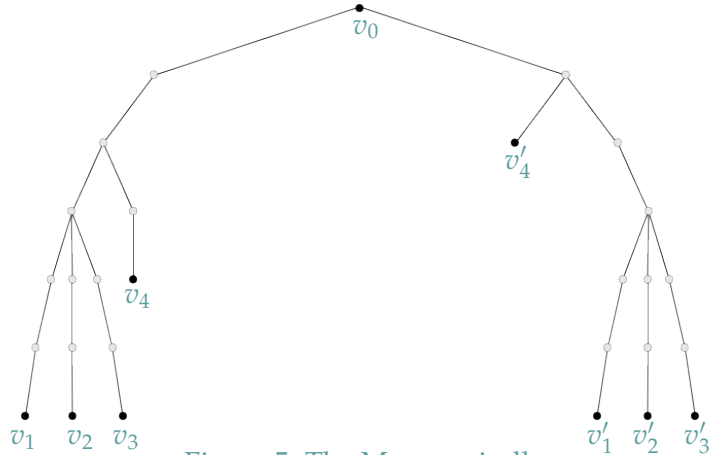
\begin{figure}[h]
{\color{linkblue}
\leavevmode \SetLabels
\L(0.49*.93) $v_0$\\
\L(0.19*-0.06) $v_1$\\
\L(0.24*-0.06) $v_2$\\
\L(0.28*-0.06) $v_3$\\
\L(0.29*.28) $v_4$\\
\L(0.695*-0.06) $v'_1$\\
\L(0.735*-0.06) $v'_2$\\
\L(0.78*-0.06) $v'_3$\\
\L(0.623*.605) $v'_4$\\
\endSetLabels
\begin{center}
\AffixLabels{\centerline{\epsfig{file =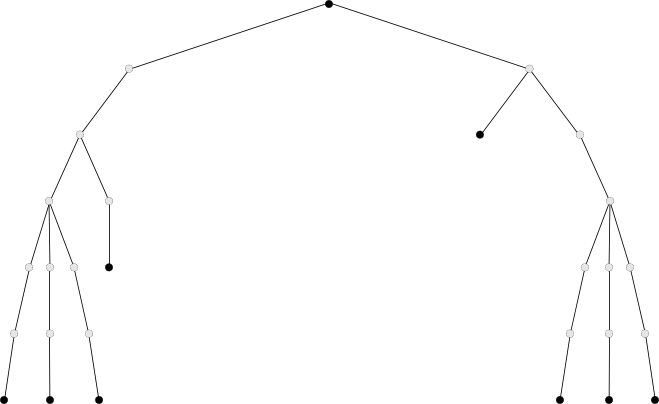,width=9.0cm,angle=0} }}
\vspace{-30pt}
\end{center}
\caption{The Moser spindle} \label{fig:MoserSpindle}
}
\end{figure}

Suppose now that it can be colored with $q$ colors. By construction, $q$ colors are needed to color the vertices $v_1,\hdots,v_q$ so $v_0$ has the same color as $v_q$. By symmetry, $v'_q$ must have the same color as $v_0$ and thus $v_q$ and $v'_q$ are the same color. This is a contradiction since $v_q$ and $v'_q$ are at distance $d$.
\end{proof}

\begin{remark}\label{rem:ammar}
By an exhaustive computer search, we checked the chromatic numbers $\chi(T_q,d)$ of certain finite subgraphs. Of what was computable, one notable result came up:
$$
\chi(T_{3},8) \geq 5.
$$
The subgraph of $T_{3}$ we considered to compute the lower bound was the graph consisting of all vertices at distance at most $8$ from a given base vertex.
\end{remark}

\section{Interval chromatic number problem}

We now focus our attention on bounding the $\Delta$-chromatic number $\chi((X,\delta),\Delta)$ when the metric space is the hyperbolic plane or a $q$-regular tree and for $\Delta := [d, cd]$ for some $d>0$ and some $c>1$. We re-use the same stratification of our spaces $\Hyp$ and $T_q$ and modify the coloring to obtain the upper bounds. The lower bounds are obtained by exhibiting cliques.

\subsection{Bounds for the hyperbolic plane}

By slightly adapting the proof of Theorem \ref{thm:chromahypupper1}, we obtain the following upper bound for large $d$. 

Note that our focus is how these bounds grow in terms of $d$ so in particular we'll use inequalities that possibly only hold for somewhat large values of $d$. Let's illustrate this by a simple example. We'll be using a bound on the $\arcsin(x)$ function. Although $\arcsin(x) > x$ for all $x>0$, they have the same behavior close to $0$, for sufficiently small $x$ we have the reverse inequality
$$
\arcsin(x) < 1.1 x.
$$

\begin{theorem}\label{thm:inthypupper}
Let $d >> 0$ be sufficiently large. Then 
$$
\chi(\Hyp, [d,cd]) < 2 (2 e^{\frac{cd-1}{2}} + 1)(cd +1).
$$
\end{theorem}

\begin{proof}
We use the checkerboard as in the bound for the $d$-chromatic number choosing $w:=d$ and $h := \log(4)$ so as to ensure that each rectangle has diameter less than $d$ for sufficiently large $d$. 

We color stratum by stratum coloring every $(\lfloor cd \rfloor +1)$th stratum with the same colors. The main difference is in how we color a stratum. This time we need $k+1$ colors to color a stratum where $k$ is the smallest integer that satisfies 
\begin{equation}\label{eqn:boundk}
k \geq e^h \sqrt{\frac{\cosh(cd)-1}{\cosh(d)-1}} = 4 \sqrt{\frac{\cosh(cd)-1}{\cosh(d)-1}} .
\end{equation}
The value $(k+1)(cd+1)$ is an upper bound. Via a small manipulation, Equation \eqref{eqn:boundk} is certainly true provided 
$$
k \geq 4 \, e^{\frac{cd-1}{2}}
$$
for large enough $d$. Thus 
$$
2 (2 e^{\frac{cd-1}{2}} + 1)(cd +1)
$$
is an upper bound.
\end{proof}

We now focus on lower bounds. To do so we will exhibit large cliques to bound $\Omega(\Hyp,[d,cd])$ from below.

\begin{theorem}\label{thm:inthyplower}
For $d>>0$ sufficiently large
$$\Omega(\Hyp,[d,cd]) > 2 \, e^{\frac{cd-1}{2}}.
$$
\end{theorem}

\begin{proof}
We choose a point $x_0 \in \Hyp$ and consider the circles $C$ of radius $\frac{c - 1}{2}d$. We now choose a maximal set of points $x_1, \hdots, x_n$ on $C$ that are successively exactly $d$ apart and $d_{\Hyp}(x_1,x_n) \geq d$. By construction the points satisfy $d_\Hyp(x_i,x_j) \in [d,cd]$ for $i,j\in \{1,\hdots,n\},$ $i\neq j$. 

We now need to estimate $n$ in function of $d$ and $c$. To do so we look at the angle $\theta$ in $x_0$ formed by a triangle $x_0, x_j, x_{j+1}$. By hyperbolic trigonometry in the triangle we have
$$
\sinh \left(\frac{d}{2}\right) = \sin\left( \frac{\theta}{2}\right)  \sinh\left(\frac{cd}{2}\right) 
$$
so 
$$
\theta = 2 \arcsin \left(\frac{\sinh\left(\frac{d}{2} \right) }{ \sinh\left(\frac{cd}{2}\right) }\right).
$$
From this
$$
n\geq \frac{2\pi}{\theta} = \frac{\pi}{\arcsin\left( \frac{\sinh\frac{d}{2} }{ \sinh\left(\frac{cd}{2}\right) } \right) }>  2 \, e^{\frac{cd-1}{2}}.
$$
\end{proof}

Obviously in the above proof, we could optimize the constant in front of the leading term but it's really the order of growth we're interested in. Put together, Theorems \ref{thm:inthypupper} and \ref{thm:inthyplower} tell us that, up to linear factor in $cd$, $\chi(\Hyp, [d,cd])$ grows like $e^{\frac{cd-1}{2}}$.
\subsection{Bounds for $k$-trees}

We begin with an upper bound which works almost identically to Theorem \ref{thm:puretree}.

\begin{theorem}\label{thm:inttreeupper}
$$\chi(T_q,[d,cd]) \leq (q-1)^{\lfloor \frac{cd}{2} +1 \rfloor} (\lfloor cd \rfloor+1)$$
\end{theorem}
\begin{proof}
The proof is very similar to the proof of Theorem \ref{thm:puretree} so we'll mainly highlight the differences. 

Using a horocyclic decomposition we color each stratum separately and reuse the colors for strata $\lfloor cd\rfloor+1$ apart. 

For a given stratum: we begin by creating bundles of vertices where vertices belong to the same bundle if they have a common root at distance $\frac{d-2}{2}$. We now create a super bundle consisting of all bundles with vertices that have a common ancestor at distance at most $\lfloor \frac{cd}{2} +1 \rfloor$. All vertices of a bundle are colored by the same color and any two bundles in a same super bundle are colored differently. This requires $(q-1)^{\lfloor \frac{cd}{2}+1 \rfloor}$ colors. 

These same colors can be used to color any other super bundle as two vertices that lie in different super bundles are at least $2 \lfloor \frac{cd}{2}+1 \rfloor > cd$ apart. 
\end{proof}

The lower bound follows the same idea as the lower bound of the corresponding theorem for the hyperbolic plane. 

\begin{theorem}\label{thm:inttreelower}
$$\Omega(T_q,[d,cd]) \geq q (q-1)^{\lfloor \frac{cd}{2} \rfloor  - \lceil\frac{d}{2} \rceil}$$
\end{theorem}

\begin{proof}
Consider a vertex $v_0$ in $T_q$ and the set $S_1$ of all vertices distance $\lfloor \frac{cd}{2} \rfloor  - \lceil\frac{d}{2} \rceil$ from $v_0$. Let $S_2$ be the set of vertices distance $\lfloor \frac{cd}{2} \rfloor  $ from $v_0$. 

Now for each vertex $v$ of $S_1$ , we associate exactly one companion vertex $v' \in S_2$ such that $v$ is on the geodesic between $v'$ and $v_0$. Denote the set of companion vertices $S$. 

Now if $v',w' \in S$ are distinct, then
$$\delta (v',w') \geq d$$
but 
$$
\delta (v',w') \leq cd.$$
Furthermore $| S | =  | S_1 |$ and as $T_q$ is $q$ regular, we have
$$
| S_1 | = q (q-1)^{\lfloor \frac{cd}{2} \rfloor  - \lceil\frac{d}{2} \rceil}
$$
as desired.
\end{proof}
\addcontentsline{toc}{section}{References}
\bibliographystyle{amsplain}
\providecommand{\bysame}{\leavevmode\hbox to3em{\hrulefill}\thinspace}
\providecommand{\MR}{\relax\ifhmode\unskip\space\fi MR }
\providecommand{\MRhref}[2]{%
  \href{http://www.ams.org/mathscinet-getitem?mr=#1}{#2}
}
\providecommand{\href}[2]{#2}

\end{document}